\documentclass[11pt,a4paper]{article}
\usepackage{amsmath,amssymb,amsthm,mathrsfs,fullpage,lineno,bbm}
\usepackage{hyperref}

\newtheorem{theorem}{Theorem}
\newtheorem{lemma}{Lemma}
\newtheorem{proposition}{Proposition}

\theoremstyle{definition}
\newtheorem{definition}{Definition}
\newtheorem{conjecture}{Conjecture}
\newtheorem{remark}{Remark}
\newtheorem{fact}{Fact}

\newcommand{\E}{\mathbb{E}}

\newcommand{\Inf}{\mathrm{Inf}}
\newcommand{\Ent}{\mathrm{Ent}}
\newcommand{\I}{\mathrm{I}}
\newcommand\tup[1]{\left\langle #1 \right\rangle}

\title{A Lower Bound for the Fourier Entropy of Boolean Functions on the Biased Hypercube}

\begin{document}
\author{Fan Chang\thanks{School of Statistics and Data Science, Nankai University, Tianjin, China and Extremal Combinatorics and Probability Group (ECOPRO), Institute for Basic Science (IBS), Daejeon, South Korea. E-mail: 1120230060@mail.nankai.edu.cn.  Supported by the NSFC under grant 124B2019 and the Institute for Basic Science (IBS-R029-C4).}}
\date{}

\maketitle

\begin{abstract}
We study Boolean functions on the $p$-biased hypercube $(\{0,1\}^n,\mu_p^n)$ through the lens of Fourier (spectral) entropy, i.e. the Shannon entropy of the squared $p$-biased Fourier coefficients.
Motivated by recent restriction-based advances on \emph{upper} bounds toward the Fourier-Entropy-Influence (FEI) conjecture, we prove a complementary, sharp \emph{lower} bound that decomposes the entropy into coordinate-wise contributions.
Let $q:=4p(1-p)$ and define $\Psi:[0,\tfrac12]\to[0,\ln 2]$ by $\Psi(t):=h\left(\frac{1+\sqrt{1-4t^2}}{2}\right)$, where $h(u):=-u\ln u-(1-u)\ln(1-u)$. We show that for every Boolean $f:(\{0,1\}^n,\mu_p^n)\to\{\pm1\}$,
\[
\Ent_p(f) \ge \sum_{k=1}^n \Psi\left(\sqrt{q(1-q)}\cdot\Inf_k^{(p)}[f]\right).
\]
When $p\neq \tfrac12$, this bound is tight and equality holds if and only if $f$ is a parity function. Our proof adapts the restriction-moment framework to the biased cube.
\end{abstract}

\section{Introduction}
Functions defined on the discrete cube $\{0,1\}^n$ are fundamental objects in Combinatorics and Theoretical Computer Science. Consider the discrete cube $\{0,1\}^n$ endowed with the $p$-biased measure $\mu_p^n=(p\delta_{\{1\}}+(1-p)\delta_{\{0\}})^{\otimes n}$ for $p\in(0,1)$. It is well known that every such function $f:(\{0,1\}^n,\mu_p^n)\to\mathbb{R}$ can be represented as a linear combination,
$$
f=\sum\limits_{S\subset[n]}\hat{f}(S)\cdot \chi^p_S,
$$
of the $2^n$ functions $\{\chi_S^p\}_{S\subset[n]}$ defined by $\chi^p_S(x)=\prod_{i\in S}\frac{x_i-p}{\sqrt{p(1-p)}}$. This representation is known as the \emph{Fourier expansion} of the function $f$, and the real numbers $\hat{f}(S)$ are known as its \emph{Fourier coefficients w.r.t. $\mu_p^n$}. Many applications (e.g., to the study of percolation~\cite{BKS1999}, threshold phenomena in random graphs~\cite{KLMM2024globalhyper2024,Talagrand1994randomgraph}, and hardness of approximation~\cite{DS2005Annal}) rely upon the use of the biased measure on the discrete cube.

For Boolean-valued functions $f:(\{0,1\}^n,\mu_p^n)\to\{\pm1\}$, Parseval's identity gives $\sum_{S\subset[n]}\hat{f}(S)^2=1$, so the squared Fourier coefficients form a probability distribution on $\{0,1\}^n$. This motivates the following notion of \emph{spectral entropy} (all logarithms are natural throughout): the Shannon entropy of this distribution, which quantitatively captures how widely the Fourier mass is dispersed among the coefficients.
\begin{definition}
For $f:(\{0,1\}^n,\mu_p^n)\to \{\pm 1\}$, the \emph{spectral/Fourier entropy of $f$ with respect to the measure $\mu_p^n$} is
$$
{\rm Ent}_p(f):=\sum\limits_{S\subset[n]}\hat{f}(S)^2\log\left(\frac{1}{\hat{f}(S)^2}\right).
$$
\end{definition}
A central theme in the analysis of Boolean functions is the interplay between spectral quantities and combinatorial sensitivity parameters such as \emph{influences}. For $i\in[n]$, let ${\rm Inf}_i^{(p)}[f]:=\mathbb{P}_{x\sim\mu_p}[f(x)\neq f(x\oplus e_i)]$ denote the $p$-biased influence of coordinate $i$, and let ${\rm I}^{(p)}[f]:=\sum_{i=1}^n{\rm Inf}_i^{(p)}[f]$ be the total influence. In the direction of \emph{upper} bounds on Fourier entropy, it is natural to recall the original Fourier-Entropy-Influence (FEI) Conjecture of Friedgut and Kalai in 1996, one of the longstanding and influential open problems in the field:

\begin{conjecture}[Friedgut--Kalai~\cite{FKthreshold1996}, FEI for the uniform cube]
There exists a universal constant $C>0$ such that for every $n$ and every Boolean function $f:(\{0,1\}^n,\mu_{1/2}^n)\to\{\pm 1\}$,
\[
{\rm Ent}_{1/2}(f) \le C\cdot {\rm I}^{(1/2)}[f].
\]
\end{conjecture}
A weaker conjecture is the so-called Fourier-Min-Entropy-Influence Conjecture which asks if the min-entropy $\min_{S\subset[n]}\log_2\hat{f}(S)^2$ could be also bounded by constant times of ${\rm I}^{(1/2)}[f]$. Keller, Mossel and Schlank~\cite{KMS2012} later formulated the following $p$-biased generalization on $(\{0,1\}^n,\mu_p^n)$:

\begin{conjecture}[Keller--Mossel--Schlank~\cite{KMS2012}, FEI for the biased cube]
 There exists a universal constant $C$, such that for any $0 < p < 1$, for any $n$ and for any Boolean-valued function $f:(\{0,1\}^n,\mu_p^n)\to \{\pm 1\}$,
 \[
 {\rm Ent}_p(f)\le Cp\log\left(\frac{1}{p}\right)\cdot{\rm I}^{(p)}[f].
 \]
\end{conjecture}
This upper bound is known to be tight (up to constants) for several natural monotone graph properties at their critical thresholds. For example, fix $r<\log n$ and choose $p_c$ so that $\binom{n}{r}\cdot p_c^{\binom{r}{2}}=\tfrac{1}{2}$. Let $f$ be the indicator function of the property ``$G$ contains $K_r$ as an induced subgraph'' on $G(n,p_c)$. Then one has
\[
{\rm Ent}_{p_c}(f)\ge Cp_c\log\frac{1}{p_c}\cdot {\rm I}^{(p_c)}[f],
\]
witnessing tightness of the conjectured scaling.

Despite substantial effort, unconditional FEI-type bounds are known only for special classes of functions
(e.g.\ \cite{ACKSW21fourierentropy,CKLS2016,das2011entropy,KLW2010,OT2013,OWZ2011,shalev2018fourier,WWW2014}).
Recently, two notable advances for the \emph{uniform} cube were obtained via random restrictions:
Kelman--Kindler--Lifshitz--Minzer--Safra~\cite{KKLMS2020} proved an $O(\I[f]\log\I[f])$ bound for the min-entropy under constant variance,
and Han~\cite{Han2025} proved an $O\big(\I[f]+\sum_i \Inf_i[f]\log(1/\Inf_i[f])\big)$ bound for the Shannon entropy.

\subsection{Main result: a sharp lower bound}\label{subsec:main-lb}

Motivated by Han's restriction-moment framework, we take a complementary direction:
rather than upper bounding $\Ent_p(f)$ in terms of total influence $\I^{(p)}[f]$, we establish a general \emph{lower bound}
on spectral entropy in terms of coordinate influences.
The guiding message is a spectral anti-concentration principle:
\emph{if $f$ has substantial sensitivity to coordinate resampling, then its Fourier spectrum cannot be too concentrated and must carry nontrivial entropy.}
In particular, our bound is expressed through a universal one-dimensional transform of each influence, and is sharp with a complete extremizer characterization.

Write $q:=4p(1-p)\in(0,1],h(q):=-q\ln q-(1-q)\ln(1-q)$, and define, for $t\in[0,\tfrac12]$,
\begin{equation}\label{eq:Psi-intro}
\Psi(t):=h\left(\frac{1+\sqrt{1-4t^2}}{2}\right).
\end{equation}
Equivalently, if $\theta\in[0,1]$ and $t=\sqrt{\theta(1-\theta)}$, then $h(\theta)=\Psi(t)$, so $\Psi$
is simply the binary entropy expressed in terms of the Bhattacharyya parameter $2t$~\cite{kesal2017generalized}.
Note also that $q=1$ iff $p=\tfrac12$, and then $q(1-q)=0$ and $h(q)=0$; in particular, any influence-based
\emph{uniform} lower bound must degenerate at the uniform point. This is unavoidable:
already on the uniform cube, parity functions have $\Ent_{1/2}(f)=0$ while $\Inf_i^{(1/2)}[f]=1$ on their support.

Our main theorem is the following sharp non-linear lower bound.

\begin{theorem}\label{thm:main-1}
For every Boolean $f:(\{0,1\}^n,\mu_p^n)\to\{\pm1\}$ and every $0<p<1$,
\begin{equation}\label{eq:main-Psi}
\Ent_p(f) \ge \sum_{k=1}^n \Psi\left(\sqrt{q(1-q)}\cdot\Inf_k^{(p)}[f]\right).
\end{equation}
Moreover, if $p\neq\tfrac12$ then equality holds in \eqref{eq:main-Psi} if and only if
\[
f(x)=\pm\prod_{i\in T}(2x_i-1)
\]
for some subset $T\subseteq[n]$ (including $T=\varnothing$, which gives constant functions).
\end{theorem}
\begin{remark}
For small $t$, one has $\Psi(t)=t^2\log(1/t^2)+O(t^2)$, so \eqref{eq:main-Psi} yields a logarithmic strengthening
in regimes where all influences are small and spread out (e.g.\ majority-type functions).
At the other extreme, $\Psi(\sqrt{q(1-q)}\,t)$ saturates at $h(q)$ as $t\to1$, matching the parity extremizers.
Thus \eqref{eq:main-Psi} cleanly interpolates between ``many tiny influences'' and ``few large influences''.   
\end{remark}
Define $\psi(s):=\Psi(\sqrt{s})$ for $s\in[0,\tfrac14]$. Then $\psi$ is concave with $\psi(0)=0$,
and the chord bound (Lemma~\ref{lem:Psi-concave}) implies, for all $t\in[0,1]$,
\[
\Psi\left(\sqrt{q(1-q)}\cdot t\right) \ge h(q)\cdot t^2,
\]
which immediately yields the clean quadratic bound
\begin{equation}\label{eq:main-linear-intro}
\Ent_p(f) \ge h(q)\cdot\sum_{k=1}^n \Inf_k^{(p)}[f]^2.
\end{equation}
More generally, if $\max_k \Inf_k^{(p)}[f]\le\rho$, the same concavity argument gives a strictly better constant
\[
\Ent_p(f) \ge C(p,\rho)\cdot\sum_{k=1}^n \Inf_k^{(p)}[f]^2,
\qquad
C(p,\rho):=\frac{\Psi(\sqrt{q(1-q)}\cdot\rho)}{\rho^2},
\]
with $C(p,\rho)\ge h(q)$ and strict inequality when $\rho<1$.

\medskip\noindent\emph{Organization.}
This paper is organized as follows.
Section~\ref{sec:pre} collects notation and preliminaries on $p$-biased Fourier analysis, random restrictions, and basic properties of the transform $\Psi$.
Section~\ref{sec:proofs} contains the proof of Theorem~\ref{thm:main-1} (the main entropy lower bound) and the characterization of the equality cases.
Appendices~\ref{app:convexityof Phi} and~\ref{app:convexityof phi} provide proofs of Lemmas~\ref{lem:Psi-convex} and~\ref{lem:Psi-concave}, respectively.

\section{Preliminaries and Proofs}\label{sec:pre}

\subsection{Biased Fourier analysis}
We work on $L^2(\{0,1\}^n,\mu_p^n)$, where $\mu_p$ is the distribution on $\{0,1\}$ defined by $\mu_p(1)=p, \mu_p(0)=1-p$. For $x\in \{0,1\}^n$, let $|x|=|\{i\in[n]:x_i=1\}|$ and then
\[
\mu_p(x):=\mu_p^{n}(x)=p^{|x|}(1-p)^{n-|x|}, \quad \underset{x\sim\mu_p}{\mathbb{E}}[f(x)]=\sum\limits_{x\in \{0,1\}^n}f(x)\mu_p(x).
\]
In the context of $p$-biased Fourier analysis we define the $p$-biased basis function $\chi_i^p:\{0,1\}\to\mathbb{R}$ by
$$
\chi_i^p(x_i):=\frac{x_i-\mu}{\sigma}, \ \mu=\underset{x_i\sim\mu_p}{\mathbb{E}}[x_i]=p, \ \sigma=\sqrt{\underset{x_i\sim\mu_p}{\mathbb{E}}[x_i^2]-\underset{x_i\sim\mu_p}{\mathbb{E}}[x_i]^2}=\sqrt{p-p^2}=\sqrt{p(1-p)}.
$$
Equivalently,
\[
\chi_i^p(x_i)=
\begin{cases}
\sqrt{\frac{1-p}{p}}, &\text{if $x_i=1$},\\
-\sqrt{\frac{p}{1-p}}, &\text{if $x_i=0$}.
\end{cases}
\]
For $S\subset[n]$ we define the product $p$-biased Fourier basis functions $\chi_S^p(x)=\prod_{i\in S}\chi_i^p(x_i).$ By construction $\{\chi_S^p\}_{S\subseteq[n]}$ is an orthonormal basis of $L^2(\{0,1\}^n,\mu_p^n)$. For $f\in L^2(\{0,1\}^n,\mu_p^n)$ we write its $p$-biased Fourier coefficients as
\[
\hat{f}(S)=\underset{x\sim\mu_p^n}{\mathbb{E}}[f(x)\cdot \chi_S^p(x)],
\]
so that
\[
f(x)=\sum\limits_{S\subseteq[n]} \hat{f}(S)\chi^p_S(x).
\]
Next we have Plancherel's identity, which states that the inner product of $f$ and $g$ is precisely the dot product of their vectors of Fourier coefficients.
\begin{fact}[Plancherel's identity]
 Let $f,g\in L^2(\{0,1\}^n,\mu_p^n)$. Then $\tup{f,g}=\sum\limits_{S\subseteq[n]}\hat{f}(S)\hat{g}(S)$.   
\end{fact}

\begin{definition}
For $i\in[n]$, the $i$th (discrete) derivative operator $\partial_i$ on $L^2(\{0,1\}^n,\mu_p^n)$ is defined by
\[
\partial_i f(x)=\sigma\cdot (f_{i\to 1}(x)-f_{i\to 0}(x)).
\]
\end{definition}
\begin{fact}
Let $f\in L^2(\{0,1\}^n,\mu_p^n)$, and let $i\in[n]$. Then
$$
\partial_if(x)=\sum\limits_{S:i\in S}\hat{f}(S)\chi^p_{S\setminus\{i\}}(x).
$$
\end{fact}

\begin{definition}
For a function $f:(\{0,1\}^n,\mu_p^n)\to \{\pm1\}$, the $i$th-influence of $f$ is defined by
$$
{\rm Inf}^{(p)}_i[f]:=\underset{x\sim \mu_p}{\mathbb{P}}[f(x)\neq f(x\oplus e_i)],
$$
where $x\oplus e_i$ is obtained from $x$ by flipping the $i$th coordinate. The total influence of the function is ${\rm I}^{(p)}[f]:=\sum_{i=1}^n{\rm Inf}^{(p)}_i[f]$.
\end{definition}
The derivative gives the standard Fourier expansion of influences:
\[
\sum\limits_{S:i\in S}\hat{f}(S)^2=\underset{\mu_p}{\mathbb{E}}[|\partial_if|^2]=\sigma^2\cdot\underset{\mu_p}{\mathbb{E}}\left[\left|f_{i\to 1}-f_{i\to0}\right|^2\right]=4\sigma^2\cdot\underset{x\sim \mu_p}{\mathbb{P}}[f(x)\neq f(x\oplus e_i)]=4p(1-p)\cdot{\rm Inf}_i^{(p)}[f],
\] 
and summing over $i$ gives
\[
{\rm I}^{(p)}[f]:=\sum_{i=1}^n{\rm Inf}_i^{(p)}[f]=\frac{1}{4p(1-p)}\sum\limits_{i=1}^n\sum\limits_{S:i\in S}\hat{f}(S)^2=\frac{1}{4p(1-p)}\sum\limits_{S\subset [n]}|S|\hat{f}(S)^2.
\]

\begin{lemma}\label{lem:Fourierp-biasedinf}
For any $f:(\{0,1\}^n,\mu_p^n)\to\mathbb{R}$ and any $k\in[n]$,
\begin{equation}
\sum_{S\subset[n]\setminus\{k\}}\hat f(S)\hat f(S\cup\{k\})
=\tup{f,\partial_k(f)}.
\end{equation}
\end{lemma}
\begin{proof}
Since $\partial_k f=\sum_{S\ni k}\hat f(S)\chi_{S\setminus\{k\}}^p$, orthonormality gives
\begin{equation*}
    \begin{split}
\tup{f,\partial_k f}&=\tup{\sum_{T\subset[n]} \hat f(T)\chi_T^p,\sum_{S\ni k}\hat f(S)\chi_{S\setminus\{k\}}^p}\\
&=\sum_{T\subset[n]} \sum_{S\ni k}\hat{f}(T)\hat{f}(S)\tup{\chi_T^p,\chi_{S\setminus\{k\}}^p}=\sum_{S\subset[n]\setminus\{k\}}\hat f(S)\hat f(S\cup\{k\}).        
    \end{split}
\end{equation*}

\end{proof}

\begin{lemma}\label{lem:p-biasedInf2}
Let $f:(\{0,1\}^n,\mu_p^n)\to\{\pm1\}$. Then for every $k\in[n]$,
\[
\tup{f,\partial_k(f)}=2\sqrt{p(1-p)}(2p-1)\cdot{\rm Inf}^{(p)}_k[f].
\]
\end{lemma}

\begin{proof}
Fix $x_{-k}$ and set $\alpha:=f(1,x_{-k})$, $\beta:=f(0,x_{-k})$.
Then
\begin{equation*}
\begin{split}
\tup{f,\partial_k(f)}&=\underset{x_{-k}}{\mathbb{E}}\underset{x_{k}}{\mathbb{E}}\left[f(x_k,x_{-k})\cdot \sqrt{p(1-p)}(\alpha-\beta)\right]\\
&=\sqrt{p(1-p)}\underset{x_{-k}}{\mathbb{E}}\left[(p\alpha+(1-p)\beta)\cdot(\alpha-\beta)\right]=\sqrt{p(1-p)}\cdot(2p-1)\underset{x_{-k}}{\mathbb{E}}[1-\alpha\beta]\\
&=\sqrt{p(1-p)}\cdot(2p-1)\underset{x_{-k}}{\mathbb{E}}[2\mathbbm{1}_{\alpha\neq \beta}]=2\sqrt{p(1-p)}\cdot(2p-1)\underset{x_{-k}}{\mathbb{P}}[\alpha\neq \beta]\\
&=2\sqrt{p(1-p)}\cdot(2p-1)\cdot{\rm Inf}_k^{(p)}[f],
\end{split}
\end{equation*}
where we use $\alpha^2=\beta^2=1$. 
\end{proof}

\subsection{Random restrictions technique}

\noindent One of the most useful tools in analysis of Boolean functions is the notion of restrictions and random restrictions. Let $J\subseteq[n]$ be a subset of coordinates thought of as ``alive'' and coordinates in $J^c=[n]\setminus J$ thought of as ``restricted''. 
\begin{definition}[Restrictions]
Suppose we have a function $f:(\{0,1\}^n,\mu_p^n)\to\mathbb{R}$, a set of coordinates $J\subseteq[n]$ and an assignment to them $z\in\{0,1\}^{J^c}$. The restricted function $f_{J^c\to z}:(\{0,1\}^J,\mu_p^J)\to\mathbb{R}$ is defined by
$$
f_{J^c\to z}(y)=f(x_J=y,x_{J^c}=z)=f(y,z).
$$
\end{definition}
\noindent First, we give a formula for the Fourier coefficients of the restricted function.

\begin{proposition}\label{prop:RFour1}
Let $f:(\{0,1\}^n,\mu_p^n)\to\mathbb{R}, J\subseteq[n], z\in\{0,1\}^{J^c}$ and $T\subseteq J$. Then we have
$$
\hat{f}_{J^c\to z}(T)=\sum\limits_{S\subseteq J^c}\hat{f}(S\cup T)\chi^p_S(z).
$$
\end{proposition}
\begin{proof}
\noindent We write $f$ according to its Fourier transform, decomposing a character into its $J$ and $J^c$ parts
$$
f(x)=\sum\limits_{T\subseteq J,S\subseteq J^c}\hat{f}(S\cup T)\chi^p_{S\cup T}(x)=\sum\limits_{T\subseteq J,S\subseteq J^c}\hat{f}(S\cup T)\chi^p_T(x_J)\chi^p_S(x_{J^c}).
$$
Plugging in the value $y$ to $x_J$ and $z$ to $x_{J^c}$, we get that
$$
f_{J^c\to z}(y)=f(y,z)=\sum\limits_{T\subseteq J}\left(\sum\limits_{S\subseteq J^c}\hat{f}(S\cup T)\chi^p_S(z)\right)\chi^p_T(y)=\sum\limits_{T\subseteq J}\hat{f}_{J^c\to z}(T)\chi^p_T(y).
$$
The fact now follows from the uniqueness of the Fourier decomposition.
\end{proof}

\begin{definition}[Random restrictions]
Given $f:(\{0,1\}^n,\mu_p^n)\to\mathbb{R}$ and $J\subseteq[n]$, a random restriction of $f$ on $J$ is a function $f_{J^c\to z}$ wherein $z\in\{0,1\}^{J^c}$ is sampled from distribution $\mu_p^{J^c}$.
\end{definition}

\begin{proposition}\label{prop:RFour2}
Let $f:(\{0,1\}^n,\mu_p^n)\to\mathbb{R}, J\subseteq[n]$ and $T\subseteq J$. Then we have
$$
\underset{z\sim\mu_p^{J^c}}{\mathbb{E}}\left[\hat{f}_{J^c\to z}(T)\right]=\hat{f}(T).
$$
\end{proposition}
\begin{proof}
By Proposition~\ref{prop:RFour1} and linearity of expectation
$$
\underset{z\sim\mu_p^{J^c}}{\mathbb{E}}\left[\hat{f}_{J^c\to z}(T)\right]=\underset{z\sim\mu_p^{J^c}}{\mathbb{E}}\left[\sum\limits_{S\subseteq J^c}\hat{f}(S\cup T)\chi^p_S(z)\right]=\sum\limits_{S\subseteq J^c}\hat{f}(S\cup T)\underset{z\sim\mu_p^{J^c}}{\mathbb{E}}[\chi^p_S(z)]=\hat{f}(T),
$$
since $\mathbb{E}_z[\chi_S^p(z)]=0$ for $S\neq\emptyset$ and equals $1$ for $S=\emptyset$.
\end{proof}
\subsection{Entropy as a function of the Bhattacharyya parameter}
Define, for $t\in[0,\tfrac12]$,
\begin{equation}\label{eq:Psi-def}
\Psi(t):=h\left(\frac{1+\sqrt{1-4t^2}}{2}\right).
\end{equation}
Note that if $\theta\in[0,1]$ and $t=\sqrt{\theta(1-\theta)}$, then $h(\theta)=\Psi(t)$.
\begin{lemma}[Convexity of $\Psi$]\label{lem:Psi-convex}
$\Psi:[0,\tfrac12]\to[0,\ln 2]$ is increasing and convex.
\end{lemma}
\begin{proof}
See Appendix~\ref{app:convexityof Phi}.
\end{proof}
\begin{lemma}\label{lem:sharp-jensen}
Let $\{(a_i,b_i)\}_{i\in I}$ be finitely many real pairs and set $w_i:=a_i^2+b_i^2$.
If $\sum_{i\in I} w_i=1$, then
\[
\sum_{i\in I} w_ih\left(\frac{a_i^2}{a_i^2+b_i^2}\right)
\ge
\Psi\left(\left|\sum_{i\in I} a_i b_i\right|\right).
\]
\end{lemma}

\begin{proof}
For indices with $w_i>0$, define $\theta_i=\frac{a_i^2}{w_i}$ and $t_i=\sqrt{\theta_i(1-\theta_i)}=\frac{|a_ib_i|}{w_i}\in[0,\tfrac12]$.
Then $h(\theta_i)=\Psi(t_i)$ and
\[
\sum_i w_i h(\theta_i)=\sum_i w_i \Psi(t_i) \ge \Psi\left(\sum_i w_i t_i\right)
=\Psi\left(\sum_i |a_i b_i|\right)
\ge \Psi\left(\left|\sum_i a_i b_i\right|\right),
\]
using Jensen with Lemma~\ref{lem:Psi-convex}, then $\sum_i|a_ib_i|\ge|\sum_i a_ib_i|$ and monotonicity of $\Psi$.
\end{proof}
\begin{lemma}[Proposition~3.3 in~\cite{wyner2003common}]\label{lem:Psi-concave}
Define $\psi(s):=\Psi(\sqrt{s})$ for $s\in[0,\tfrac14]$. Then $\psi$ is concave on $[0,\tfrac14]$ and $\psi(0)=0$.
Consequently, for every $s_0\in(0,\tfrac14]$ and every $s\in[0,s_0]$,
\begin{equation}\label{eq:chord}
\Psi(\sqrt{s})=\psi(s) \ge\frac{\psi(s_0)}{s_0}\cdot s.
\end{equation}
\end{lemma}
\begin{proof}
This is a standard fact; for completeness we provide a proof in Appendix~\ref{app:convexityof Phi}.
\end{proof}

\section{Proofs}\label{sec:proofs}

\subsection{Proof strategy}
We outline the main ideas before the technical lemmas.
\begin{itemize}
    \item[(1)] (\emph{Entropy as a derivative of a restriction moment.}) We encode the Fourier entropy as the $\varepsilon$-derivative at $0$ of a restricted-moment functional
    $M_{J,\varepsilon,p}(f)$, and telescope along a fixed coordinate chain
    $\varnothing=J_0\subset J_1\subset\cdots\subset J_n=[n]$.

    \item[(2)] (\emph{One coordinate reveals a two-point mixture.})
  Adding a new coordinate $k$ pairs restricted coefficients $(a,b)=(\widehat{f_{J_k^c\to z'}}(S),\widehat{f_{J_k^c\to z'}}(S\cup\{k\}))$, and the corresponding coefficient on the smaller alive set becomes $a+\chi_k^p(z_k)\,b$. This yields a two-point functional $\Phi_{p,\varepsilon}(a,b)$ for the one-step increment.

    \item[(3)] (\emph{Main technical step.})
    Differentiating $\Phi_{p,\varepsilon}$ at $\varepsilon=0$ yields a local ``entropy drop'' governed by the binary entropy function:
    \[
    -\partial_\varepsilon \Phi_{p,\varepsilon}(a,b)\big|_{\varepsilon=0}
\ge(a^2+b^2)h\left(\frac{a^2}{a^2+b^2}\right).
    \]
    We then introduce the convex increasing transform
    $\Psi(t):=h\left(\frac{1+\sqrt{1-4t^2}}{2}\right)$, so that $h(\theta)=\Psi(\sqrt{\theta(1-\theta)})$.
    A sharp Jensen step converts the weighted sum over $S$ into a single term:
    \[
    \sum_{S}(a_S^2+b_S^2)h\left(\frac{a_S^2}{a_S^2+b_S^2}\right)
\ge\Psi\left(\left|\sum_{S} a_S b_S\right|\right),
    \]
    which is the key point where we avoid quadratic relaxations and retain the logarithmic structure.

    \item[(4)] (\emph{Identify influences and sum over coordinates.}) We show $\mathbb{E}_{z'}[\sum_S a_S b_S]=\langle f,\partial_k f\rangle$ and $|\langle f,\partial_k f\rangle|=\sqrt{q(1-q)}\,\Inf_k^{(p)}[f]$. 
    Plugging this into the previous step and summing over $k$ yields the main non-linear bound
    \[
    \Ent_p(f) \ge\sum_{k=1}^n \Psi\left(\sqrt{q(1-q)}\cdot\Inf_k^{(p)}[f]\right).
    \]
\end{itemize}

\subsection{Restricted moments and One-step increment as a two-point functional}
For $J\subseteq[n]$ and $\varepsilon\in[0,\frac12)$ define
\[
M_{J,\varepsilon,p}(f):=\sum_{S\subset J}\underset{z\sim\mu_p^{J^c}}{\mathbb{E}}
\left[\left|\widehat{f_{J^c\to z}}(S)\right|^{2(1+\varepsilon)}\right].
\]
Fix a coordinate chain $\varnothing=J_0\subset J_1\subset\cdots\subset J_n=[n]$ with $J_k\setminus J_{k-1}=\{k\}$, and set
\[
\Delta_{k,\varepsilon,p}(f):=M_{J_k,\varepsilon,p}(f)-M_{J_{k-1},\varepsilon,p}(f).
\]
Then
\begin{equation}\label{eq:telescope}
M_{[n],\varepsilon,p}(f)
=
M_{\varnothing,\varepsilon,p}(f)+\sum_{k=1}^n \Delta_{k,\varepsilon,p}(f).
\end{equation}
For Boolean $f$, $M_{\varnothing,\varepsilon,p}(f)=\mathbb{E}[|f|^{2(1+\varepsilon)}]=1$, while
$M_{[n],\varepsilon,p}(f)=\sum_{S\subset[n]}|\hat f(S)|^{2(1+\varepsilon)}$.
Differentiating \eqref{eq:telescope} at $\varepsilon=0$ yields
\begin{equation}\label{eq:Ent-derivative}
\Ent_p(f)
=
-\frac{d}{d\varepsilon}M_{[n],\varepsilon,p}(f)\Big|_{\varepsilon=0}
=
-\sum_{k=1}^n \frac{d}{d\varepsilon}\Delta_{k,\varepsilon,p}(f)\Big|_{\varepsilon=0}.
\end{equation}
Fix $k\in[n]$ and abbreviate $J_1:=J_{k-1}$ and $J_2:=J_k=J_1\cup\{k\}$.
Write $z=(z',z_k)\in\{0,1\}^{J_1^c}$ with $z'\in\{0,1\}^{J_2^c}$ and $z_k\in\{0,1\}$, and let $\alpha:=\sqrt{\frac{1-p}{p}},\beta:=\sqrt{\frac{p}{1-p}}$.
\begin{lemma}\label{lem:pairing}
For every $S\subseteq J_1$ and every $z=(z',z_k)$,
\[
\widehat{f_{J_1^c\to z}}(S)
=
\widehat{f_{J_2^c\to z'}}(S)
+\chi_k^p(z_k)\widehat{f_{J_2^c\to z'}}(S\cup\{k\}).
\]
\end{lemma}

\begin{proof}
Expand $f_{J_2^c\to z'}$ on the alive set $J_2$:
\[
f_{J_2^c\to z'}(y,z_k)
=\sum_{T\subseteq J_2}\widehat{f_{J_2^c\to z'}}(T)\chi_T^p(y,z_k)
=\sum_{S\subseteq J_1}\Big(a_S+\chi_k^p(z_k)b_S\Big)\chi_S^p(y),
\]
where $a_S=\widehat{f_{J_2^c\to z'}}(S)$ and $b_S=\widehat{f_{J_2^c\to z'}}(S\cup\{k\})$.
Identifying the coefficient of $\chi_S^p$ in the restriction $f_{J_1^c\to z}(y)=f_{J_2^c\to z'}(y,z_k)$ gives the claim.
\end{proof}

Define the two-point functional, for $a,b\in\mathbb{R}$,
\begin{equation}\label{eq:Phi-def}
\Phi_{p,\varepsilon}(a,b)
:=
|a|^{2(1+\varepsilon)}+|b|^{2(1+\varepsilon)}
-
p\,|a+\alpha b|^{2(1+\varepsilon)}
-(1-p)\,|a-\beta b|^{2(1+\varepsilon)}.
\end{equation}

\begin{lemma}[Derivative bound via convexity]\label{lem:Phi-derivative}
For all real $a,b$,
\begin{equation}\label{eq:Phi-derivative}
\frac{\partial}{\partial\varepsilon}\Phi_{p,\varepsilon}(a,b)\Big|_{\varepsilon=0}
\le
-(a^2+b^2)\cdot h\left(\frac{a^2}{a^2+b^2}\right),
\end{equation}
interpreting the right-hand side as $0$ when $a=b=0$.
\end{lemma}
\begin{proof}
Let $u:=(a+\alpha b)^2$ and $v:=(a-\beta b)^2$. Differentiating \eqref{eq:Phi-def} at $\varepsilon=0$ gives
\[
\frac{\partial}{\partial\varepsilon}\Phi_{p,\varepsilon}(a,b)\Big|_{\varepsilon=0}
=
a^2\log a^2+b^2\log b^2 -p\,u\log u-(1-p)\,v\log v,
\]
with $0\log0:=0$.
Since $t\mapsto t\log t$ is convex, Jensen yields
\[
p\,u\log u+(1-p)\,v\log v \ge \left(pu+(1-p)v\right)\log\left(pu+(1-p)v\right).
\]
A direct calculation shows $pu+(1-p)v=a^2+b^2$. Substituting back gives \eqref{eq:Phi-derivative}.
\end{proof}

\begin{lemma}[Increment formula]\label{lem:increment}
For the above chain and $k$,
\begin{equation}\label{eq:Delta-Phi}
\Delta_{k,\varepsilon,p}(f)
=
\sum_{S\subset J_1}\mathbb{E}_{z'\sim\mu_p^{J_2^c}}
\left[
\Phi_{p,\varepsilon}\left(
\widehat{f_{J_2^c\to z'}}(S),
\widehat{f_{J_2^c\to z'}}(S\cup\{k\})
\right)
\right].
\end{equation}
\end{lemma}

\begin{proof}
By Lemma~\ref{lem:pairing}, conditioning on $z'$ and averaging over $z_k\sim\mu_p$ yields
\[
\mathbb{E}_{z_k}\left[\left|a_S+\chi_k^p(z_k)b_S\right|^{2(1+\varepsilon)}\right]
=
p|a_S+\alpha b_S|^{2(1+\varepsilon)}+(1-p)|a_S-\beta b_S|^{2(1+\varepsilon)}.
\]
On the other hand, $M_{J_2,\varepsilon,p}$ contains both coefficients indexed by $S$ and by $S\cup\{k\}$.
Subtracting $M_{J_1,\varepsilon,p}$ from $M_{J_2,\varepsilon,p}$ gives \eqref{eq:Delta-Phi}.
\end{proof}
Combining \eqref{eq:Ent-derivative}, \eqref{eq:Delta-Phi}, and Lemma~\ref{lem:Phi-derivative} yields
\begin{equation}\label{eq:Ent-local-entropy}
\Ent_p(f)
\ge
\sum_{k=1}^n
\sum_{S\subset J_{k-1}}
\underset{z'\sim\mu_p^{J_k^c}}{\mathbb{E}}
\left[
(a_S^2+b_S^2)\cdot h\left(\frac{a_S^2}{a_S^2+b_S^2}\right)
\right],
\end{equation}
where for brevity we write $a_S=\widehat{f_{J_k^c\to z'}}(S)$ and $b_S=\widehat{f_{J_k^c\to z'}}(S\cup\{k\})$.
\subsection{Proof of Theorem~\ref{thm:main-1}}
\begin{lemma}[Normalization]\label{lem:normalization}
Fix $k$ and $z'\in\{0,1\}^{J_k^c}$. Then
\[
\sum_{S\subset J_{k-1}}\left(a_S^2+b_S^2\right)=1,
\]
where $a_S=\widehat{f_{J_k^c\to z'}}(S)$ and $b_S=\widehat{f_{J_k^c\to z'}}(S\cup\{k\})$.
\end{lemma}

\begin{proof}
For fixed $(z',z_k)$, the restricted function $f_{J_{k-1}^c\to(z',z_k)}$ is Boolean on $\{0,1\}^{J_{k-1}}$,
hence Parseval gives $\sum_{S\subset J_{k-1}}\widehat{f_{J_{k-1}^c\to(z',z_k)}}(S)^2=1$.
Averaging over $z_k\sim\mu_p$ and using Lemma~\ref{lem:pairing} yields
$\mathbb{E}_{z_k}[\widehat{f_{J_{k-1}^c\to(z',z_k)}}(S)^2]=a_S^2+b_S^2$ for each $S$, and summing over $S$ gives the claim.
\end{proof}
\begin{lemma}\label{lem:identify}
For each fixed $k$,
\[
\underset{z'\sim\mu_p^{J_k^c}}{\mathbb{E}}\left[\sum_{S\subset J_{k-1}} a_S b_S\right]
=
\sum_{R\subset[n]\setminus\{k\}}\hat f(R)\hat f(R\cup\{k\})
=
\langle f,\partial_k f\rangle,
\]
where $a_S=\widehat{f_{J_k^c\to z'}}(S)$ and $b_S=\widehat{f_{J_k^c\to z'}}(S\cup\{k\})$.
\end{lemma}

\begin{proof}
Fix $S\subseteq J_{k-1}$. By the restriction formula (Proposition~\ref{prop:RFour1}),
\[
a_S(z')=\sum_{T\subseteq J_k^c}\hat f(S\cup T)\chi_T^p(z'),
\qquad
b_S(z')=\sum_{T\subseteq J_k^c}\hat f(S\cup T\cup\{k\})\chi_T^p(z').
\]
Taking the inner product in $L^2(\{0,1\}^{J_k^c},\mu_p^{J_k^c})$ and using orthonormality of
$\{\chi_T^p\}_{T\subseteq J_k^c}$ gives
\[
\mathbb{E}_{z'}[a_S(z')b_S(z')]
=\sum_{T\subseteq J_k^c}\hat f(S\cup T)\hat f(S\cup T\cup\{k\}).
\]
Summing over $S\subseteq J_{k-1}$ and re-indexing by $R=S\cup T$ (note that
$[n]\setminus\{k\}=J_{k-1}\sqcup J_k^c$) yields
\[
\mathbb{E}_{z'}\left[\sum_{S\subseteq J_{k-1}} a_S b_S\right]
=
\sum_{R\subseteq[n]\setminus\{k\}}\hat f(R)\hat f(R\cup\{k\}).
\]
The second equality is Lemma~\ref{lem:Fourierp-biasedinf}.
\end{proof}

\begin{proof}[Proof of Theorem~\ref{thm:main-1}]
Start from \eqref{eq:Ent-local-entropy}. Fix $k$ and $z'$, and apply Lemma~\ref{lem:sharp-jensen} to the family
$\{(a_S,b_S)\}_{S\subset J_{k-1}}$, using Lemma~\ref{lem:normalization} to verify $\sum_S(a_S^2+b_S^2)=1$:
\[
\sum_{S\subset J_{k-1}}(a_S^2+b_S^2)\cdot h\left(\frac{a_S^2}{a_S^2+b_S^2}\right)\ge
\Psi\left(\left|\sum_{S\subset J_{k-1}} a_S b_S\right|\right).
\]
Taking expectation over $z'$ and using Jensen for the convex function $\Psi$ gives
\[
\mathbb{E}_{z'}\left[\Psi\!\left(\left|\sum_S a_S b_S\right|\right)\right]\ge
\Psi\left(\mathbb{E}_{z'}\left[\left|\sum_S a_S b_S\right|\right]\right)\ge
\Psi\left(\left|\mathbb{E}_{z'}\left[\sum_S a_S b_S\right]\right|\right).
\]
By Lemma~\ref{lem:identify}, $\mathbb{E}_{z'}[\sum_S a_S b_S]=\tup{f,\partial_k f}$, hence
\begin{equation}\label{eq:Ent-Psi}
\Ent_p(f) \ge \sum_{k=1}^n \Psi\left(|\tup{f,\partial_k f}|\right)=\sum_{k=1}^n \Psi\left(\sqrt{q(1-q)}\cdot\Inf_k^{(p)}[f]\right),
\end{equation}
using Lemma~\ref{lem:p-biasedInf2}, $
|\tup{f,\partial_k f}|=
2\sqrt{p(1-p)}|2p-1|\Inf_k^{(p)}[f]
=\sqrt{q(1-q)}\Inf_k^{(p)}[f]$. 
\end{proof}

\begin{proof}[Proof of~\eqref{eq:main-linear-intro}]
Apply Lemma~\ref{lem:Psi-concave} with $s=q(1-q)\Inf_k^{(p)}[f]^2\le q(1-q)$ and $s_0=q(1-q)$.
Since $\psi(s_0)=\Psi(\sqrt{q(1-q)})=h(q)$ (using $h(q)=h(1-q)$), \eqref{eq:chord} yields
\[
\Psi\left(\sqrt{q(1-q)}\cdot\Inf_k^{(p)}[f]\right) \ge h(q)\cdot\Inf_k^{(p)}[f]^2.
\]
Plugging this into \eqref{eq:Ent-Psi} and summing over $k$ completes the proof.
\end{proof}

\subsection{Characterization of Equality Cases}  
\begin{proposition}[Extremizers for the non-linear bound]\label{prop:extremizers-nonlinear}
Assume $p\neq\tfrac12$ and write $q:=4p(1-p)\in(0,1)$.
Equality holds in
\[
\Ent_p(f)\ge \sum_{k=1}^n\Psi\left(\sqrt{q(1-q)}\cdot\Inf_k^{(p)}[f]\right)
\]
if and only if $f(x)=\pm\prod_{i\in T}(2x_i-1)$ for some $T\subseteq[n]$ (including $T=\varnothing$).
\end{proposition}
\begin{proof}
Assume first that equality holds in the stated inequality. In the proof of the bound one writes
$\Ent_p(f)=\sum_{k=1}^n \Delta_{k,\varepsilon,p}(f)$ (and then lets $\varepsilon\downarrow 0$),
and for each $k$ one proves the one-coordinate lower bound
\[
\Delta_{k,\varepsilon,p}(f) \ge \Psi\left(\sqrt{q(1-q)}\cdot\Inf_k^{(p)}[f]\right).
\]
Hence equality in the sum forces equality for each $k$ separately. Fix such a $k$.
\begin{itemize}
    \item[(1)] \emph{Equality in the two-point convexity step (Lemma~\ref{lem:Phi-derivative}).} For $z'\in\{0,1\}^{J_k^c}$ write $a_S(z'):=\widehat{f_{J_k^c\to z'}}(S)$ and $b_S(z'):=\widehat{f_{J_k^c\to z'}}(S\cup\{k\})$ for any $S\subseteq J_{k-1}$. In Lemma~\ref{lem:Phi-derivative} one applies Jensen to the convex function $t\mapsto t\log t$
to two numbers of the form $(a_S+\alpha b_S)^2$ and $(a_S-\beta b_S)^2$ (with $\alpha,\beta>0$ determined by $p$).
Equality in Jensen forces these two inputs to coincide:
\[
(a_S+\alpha b_S)^2=(a_S-\beta b_S)^2 \qquad\text{for every }S \text{ and a.e.\ } z'.
\]
Thus for each $(S,z')$ either $b_S(z')=0$ or else
\[
a_S(z')=\lambda\,b_S(z')\quad\text{where}\quad
\lambda:=\frac{\beta-\alpha}{2}=\frac{2p-1}{2\sqrt{p(1-p)}}.
\]
In particular, whenever $(a_S(z'),b_S(z'))\neq(0,0)$ we have
\[
\frac{|b_S(z')|}{|a_S(z')|}=\frac{1}{|\lambda|}
=\sqrt{\frac{q}{1-q}},\qquad
\frac{|a_S(z')b_S(z')|}{a_S(z')^2+b_S(z')^2}=\sqrt{q(1-q)}.
\]
\item[(2)] \emph{Equality in the ``sharp Jensen'' step over $S$ (Lemma~\ref{lem:sharp-jensen}).} The sharp Jensen step combines (i) the fact that all nonzero pairs $(a_S(z'),b_S(z'))$ share the same ratio
$|b_S|/|a_S|=\sqrt{q/(1-q)}$ with (ii) the inequality $\sum_S |a_S b_S|\ge \big|\sum_S a_S b_S\big|$.
Equality forces the summands $a_S(z')b_S(z')$ to have a common sign for all $S$ with $(a_S,b_S)\neq(0,0)$.
Consequently, for a.e.\ $z'$ either $b_S(z')\equiv 0$ for all $S$, or else there exists $\sigma_k(z')\in\{\pm1\}$ such that
\begin{equation}\label{eq:pointwise-ratio}
b_S(z')=\sigma_k(z')\sqrt{\frac{q}{1-q}}\; a_S(z')
\qquad\text{for all }S\subseteq J_{k-1}.
\end{equation}
Note that we use the inequality $\E_{z'}|X(z')|\ge |\E_{z'}X(z')|$ for $
X(z'):=\sum_{S\subseteq J_{k-1}} a_S(z')\,b_S(z')$. Equality in $\E|X|\ge|\E X|$ implies that $X(z')$ has an a.s. constant sign (or is a.s. zero).
Under \eqref{eq:pointwise-ratio} we have
\[
X(z')=\sigma_k(z')\sqrt{\frac{q}{1-q}}\sum_{S\subseteq J_{k-1}} a_S(z')^2,
\]
so on the event $\sum_S a_S(z')^2>0$ the sign of $X(z')$ is exactly $\sigma_k(z')$.
Therefore $\sigma_k(z')$ must be a.s. constant on $\{\sum_S a_S(z')^2>0\}$.
Denote this constant by $\sigma_k\in\{\pm1\}$.
(If instead $\sum_S a_S(z')^2=0$ a.s., then $a_S\equiv 0$ and hence also $b_S\equiv 0$ a.s.; this is the trivial case.)

Thus, in the nontrivial case we have for every $S$ the identity of functions of $z'$
\begin{equation}\label{eq:ratio-functions}
b_S(z')=\sigma_k\sqrt{\frac{q}{1-q}}a_S(z')\qquad\text{for a.e.\ }z'.
\end{equation}
\item[(3)] \emph{The global Fourier ratio coefficientwise implies $\Inf_k^{(p)}[f]\in\{0,1\}$ for every $k$.} By Lemma~\ref{lem:identify},
for each fixed $S\subseteq J_{k-1}$ we can write $a_S(z')=\sum_{T\subseteq J_k^c}\hat f(S\cup T)\,\chi_T^p(z')$ and $b_S(z')=\sum_{T\subseteq J_k^c}\hat f(S\cup T\cup\{k\})\,\chi_T^p(z')$, where $\{\chi_T^p\}_{T\subseteq J_k^c}$ is the $p$-biased orthonormal Fourier basis in the $z'$-variables.
Since \eqref{eq:ratio-functions} is an equality of functions of $z'$, uniqueness of the Fourier expansion implies
\[
\hat f(S\cup T\cup\{k\})=\sigma_k\sqrt{\frac{q}{1-q}}\;\hat f(S\cup T)
\qquad\forall\,T\subseteq J_k^c.
\]
Renaming $R:=S\cup T$, we obtain the \emph{global} coefficientwise ratio
\begin{equation}\label{eq:global-ratio}
\hat f(R\cup\{k\})=\sigma_k\sqrt{\frac{q}{1-q}}\;\hat f(R)
\qquad\forall\,R\subseteq[n]\setminus\{k\}.
\end{equation}
In the trivial case $b_S\equiv 0$ we similarly get $\hat f(R\cup\{k\})=0$ for all $R$.

If $\hat f(R\cup\{k\})\equiv 0$, then $\sum_{S\ni k}\hat f(S)^2=0$ and hence $\Inf_k^{(p)}[f]=0$.

Otherwise, by \eqref{eq:global-ratio} and Parseval,
\[
\sum_{S\ni k}\hat f(S)^2=\sum_{R\subseteq[n]\setminus\{k\}}\hat f(R\cup\{k\})^2
=\frac{q}{1-q}\sum_{R\subseteq[n]\setminus\{k\}}\hat f(R)^2.
\]
Let $A:=\sum_{R\not\ni k}\hat f(R)^2$ and $B:=\sum_{S\ni k}\hat f(S)^2$.
Then $B=\frac{q}{1-q}A$ and $A+B=1$, so $A=1-q$ and $B=q$.
Using the standard identity $\sum_{S\ni k}\hat f(S)^2=q\,\Inf_k^{(p)}[f]$, we get $\Inf_k^{(p)}[f]=1$.
Hence for every $k$ we have $\Inf_k^{(p)}[f]\in\{0,1\}$.
\end{itemize}

Let $T:=\{k\in[n]:\Inf_k^{(p)}[f]=1\}$. Since $\mu_p^n$ has full support and $f\in\{\pm1\}$, the random variable $\frac{(f(x)-f(x\oplus e_k))^2}{4}\in\{0,1\}$ has expectation $\Inf_k^{(p)}[f]$. Thus:
\[
k\notin T \implies f(x)=f(x\oplus e_k)\ \forall x,\qquad
k\in T \implies f(x)=-f(x\oplus e_k)\ \forall x.
\]
Define
\[
g(x):=f(x)\prod_{i\in T}(2x_i-1).
\]
For $k\notin T$, both factors are invariant under $x\mapsto x\oplus e_k$, so $g(x)=g(x\oplus e_k)$.
For $k\in T$, both factors flip sign under $x\mapsto x\oplus e_k$, so again $g(x)=g(x\oplus e_k)$.
Hence $g$ is invariant under flipping any coordinate, so $g$ is constant on $\{0,1\}^n$.
Therefore $f(x)=\pm\prod_{i\in T}(2x_i-1)$, proving the ``only if'' direction.

\medskip
Conversely, let $f(x)=\pm\prod_{i\in T}(2x_i-1)$.
Then flipping coordinate $i$ changes the sign of $f$ iff $i\in T$, hence
$\Inf_i^{(p)}[f]=\mathbbm 1_{\{i\in T\}}$, and so
\[
\sum_{k=1}^n \Psi\left(\sqrt{q(1-q)}\Inf_k^{(p)}[f]\right)
=|T|\Psi\left(\sqrt{q(1-q)}\right).
\]
Moreover, in the $p$-biased Fourier basis one has
\[
2x_i-1=(2p-1)\chi_\emptyset^p(x)+2\sqrt{p(1-p)}\,\chi_{\{i\}}^p(x)
=(2p-1)\chi_\emptyset^p(x)+\sqrt q\,\chi_{\{i\}}^p(x),
\]
so expanding the product gives
\[
f(x)=\pm\prod_{i\in T}\Big((2p-1)\chi_\emptyset^p(x)+\sqrt q\,\chi_{\{i\}}^p(x)\Big)
=\pm\sum_{S\subseteq T}(2p-1)^{|T|-|S|}q^{|S|/2}\,\chi_S^p(x),
\]
and therefore $\hat f(S)^2=(2p-1)^{2(|T|-|S|)}q^{|S|}
=(1-q)^{|T|-|S|}q^{|S|}$ for $S\subseteq T$, and $\hat f(S)=0$ otherwise. Hence
\begin{align*}
\Ent_p(f)
&=\sum_{S\subseteq T}\hat f(S)^2\log\frac1{\hat f(S)^2}=\sum_{r=0}^{|T|}\binom{|T|}{r}(1-q)^{|T|-r}q^r
\log\frac{1}{(1-q)^{|T|-r}q^r}\\
&=|T|\cdot h(q)=\sum_{k=1}^n \Psi\left(\sqrt{q(1-q)}\Inf_k^{(p)}[f]\right),
\end{align*}
so equality holds. This completes the proof.
\end{proof}

\appendix

\section{Proof of Lemma~\ref{lem:Psi-convex}}\label{app:convexityof Phi}
Fix $t\in(0,\tfrac12)$ and set
\[
s=s(t):=\sqrt{1-4t^2}\in(0,1),\qquad \theta=\theta(t):=\frac{1+s}{2}\in\left(\frac12,1\right).
\]
Then $\Psi(t)=h(\theta(t))$.

\medskip
\noindent\textbf{Monotonicity.}
We have $h'(u)=\ln\frac{1-u}{u}$ for $u\in(0,1)$, and
\[
s'(t)=\frac12(1-4t^2)^{-\frac{1}{2}}\cdot(-8t)=-\frac{4t}{s},\qquad 
\theta'(t)=\frac{s'(t)}{2}=-\frac{2t}{s}.
\]
Hence by the chain rule,
\[
\Psi'(t)=h'(\theta)\cdot\theta'
=\ln\left(\frac{1-\theta}{\theta}\right)\cdot\left(-\frac{2t}{s}\right).
\]
Since $\theta\in(\tfrac12,1)$ we have $\ln(\frac{1-\theta}{\theta})<0$, while $-\frac{2t}{s}<0$, so $\Psi'(t)>0$.
Thus $\Psi$ is strictly increasing on $(0,\tfrac12)$ and hence increasing on $[0,\tfrac12]$ by continuity.

\medskip
\noindent\textbf{Convexity.}
First rewrite $
\frac{1-\theta}{\theta}=\frac{1-\frac{1+s}{2}}{\frac{1+s}{2}}=\frac{1-s}{1+s}$. Therefore the above formula for $\Psi'(t)$ becomes
\begin{equation}\label{eq:Psi-prime-short}
\Psi'(t)=\frac{2t}{s}\,L(s),
\qquad
L(s):=\ln\left(\frac{1+s}{1-s}\right).
\end{equation}
Differentiate \eqref{eq:Psi-prime-short}. Using $s'(t)=-\frac{4t}{s}$ and $L'(s)=\frac{2}{1-s^2}$, we get
\begin{equation}\label{eq:Psi-second-short}
\Psi''(t)=\left(\frac{2t}{s}\right)'L(s)+\frac{2t}{s}\,L'(s)\,s'(t)
=\frac{2}{s^3}\left(\ln\frac{1+s}{1-s}-2s\right).
\end{equation}
Now define $g(s):=\ln(\frac{1+s}{1-s})-2s$ on $(0,1)$. Then $g(0)=0$ (by continuity) and
\[
g'(s)=\left(\frac{1}{1+s}+\frac{1}{1-s}\right)-2
=\frac{2}{1-s^2}-2=\frac{2s^2}{1-s^2}\ge 0.
\]
Thus $g(s)\ge 0$ for all $s\in(0,1)$, and \eqref{eq:Psi-second-short} implies $\Psi''(t)\ge 0$ for all $t\in(0,\tfrac12)$.
Therefore $\Psi$ is convex on $(0,\tfrac12)$, hence convex on $[0,\tfrac12]$ by continuity.
\section{Proof of Lemma~\ref{lem:Psi-concave}}\label{app:convexityof phi}
Let $s\in(0,\tfrac14)$ and set $t:=\sqrt{1-4s}\in(0,1)$, so that $\psi(s)=h(\frac{1+t}{2})$.
A direct computation yields
\[
\psi''(s)
=
-\frac{2}{t^3}\left(\frac{2t}{1-t^2}-\ln\frac{1+t}{1-t}\right) \le 0,
\]
since $\mathrm{artanh}(t)=\int_0^t \frac{du}{1-u^2}\le \frac{t}{1-t^2}$ implies
$\ln\frac{1+t}{1-t}=2\,\mathrm{artanh}(t)\le \frac{2t}{1-t^2}$.

Note that $\psi(0)=\Psi(0)=h(1)=0$. Fix $s_0\in(0,\tfrac14]$ and consider the secant line between $(0,\psi(0))$ and $(s_0,\psi(s_0))$:
\[
\ell(s):=\psi(0)+\frac{\psi(s_0)-\psi(0)}{s_0}s=\frac{\psi(s_0)}{s_0}s.
\]
Since $\psi$ is concave on $[0,s_0]$, it lies above every secant line, hence $\psi(s)\ge \ell(s)$ for all $s\in[0,s_0]$.
This is exactly \eqref{eq:chord}.

\medskip\noindent\textbf{Acknowledgements.}
The author thanks his advisor Lei Yu for valuable guidance and for pointing out the reference~\cite{wyner2003common}.

\bibliographystyle{abbrv}
\bibliography{reference}

@article {Han2025,
    AUTHOR = {Han, Xiao},
     TITLE = {A new bound for the {F}ourier-entropy-influence conjecture},
   JOURNAL = {Combinatorica},
  FJOURNAL = {Combinatorica. An International Journal on Combinatorics and
              the Theory of Computing},
    VOLUME = {45},
      YEAR = {2025},
    NUMBER = {1},
     PAGES = {Paper No. 4, 12},
      ISSN = {0209-9683},
   MRCLASS = {94A17 (94D10)},
  MRNUMBER = {4848744},
       DOI = {10.1007/s00493-024-00133-z},
       URL = {https://doi.org/10.1007/s00493-024-00133-z},
}

@article {KMS2012,
    AUTHOR = {Keller, Nathan and Mossel, Elchanan and Schlank, Tomer},
     TITLE = {A note on the entropy/influence conjecture},
   JOURNAL = {Discrete Math.},
  FJOURNAL = {Discrete Mathematics},
    VOLUME = {312},
      YEAR = {2012},
    NUMBER = {22},
     PAGES = {3364--3372},
      ISSN = {0012-365X},
   MRCLASS = {05C80},
  MRNUMBER = {2961226},
       DOI = {10.1016/j.disc.2012.07.031},
       URL = {https://doi.org/10.1016/j.disc.2012.07.031},
}

@article {BKS1999,
    AUTHOR = {Benjamini, Itai and Kalai, Gil and Schramm, Oded},
     TITLE = {Noise sensitivity of {B}oolean functions and applications to percolation},
   JOURNAL = {Inst. Hautes \'{E}tudes Sci. Publ. Math.},
  FJOURNAL = {Institut des Hautes \'{E}tudes Scientifiques. Publications
              Math\'{e}matiques},
    NUMBER = {90},
      YEAR = {1999},
     PAGES = {5--43 (2001)},
      ISSN = {0073-8301},
   MRCLASS = {60B15 (60K35 68Q15 82B43 94C10)},
  MRNUMBER = {1813223},
MRREVIEWER = {H. Kesten},
       URL = {http://www.numdam.org/item?id=PMIHES_1999__90__5_0},
}

@article {DS2005Annal,
    AUTHOR = {Dinur, Irit and Safra, Samuel},
     TITLE = {On the hardness of approximating minimum vertex cover},
   JOURNAL = {Ann. of Math. (2)},
  FJOURNAL = {Annals of Mathematics. Second Series},
    VOLUME = {162},
      YEAR = {2005},
    NUMBER = {1},
     PAGES = {439--485},
      ISSN = {0003-486X},
   MRCLASS = {68Q17 (68Q15)},
  MRNUMBER = {2178966},
MRREVIEWER = {Johan H\aa stad},
       DOI = {10.4007/annals.2005.162.439},
       URL = {https://doi.org/10.4007/annals.2005.162.439},
}

@article {KLMM2024globalhyper2024,
    AUTHOR = {Keevash, Peter and Lifshitz, Noam and Long, Eoin and Minzer,
              Dor},
     TITLE = {Hypercontractivity for global functions and sharp thresholds},
   JOURNAL = {J. Amer. Math. Soc.},
  FJOURNAL = {Journal of the American Mathematical Society},
    VOLUME = {37},
      YEAR = {2024},
    NUMBER = {1},
     PAGES = {245--279},
      ISSN = {0894-0347},
   MRCLASS = {06E30 (94D10)},
  MRNUMBER = {4654613},
MRREVIEWER = {I. Villa},
       DOI = {10.1090/jams/1027},
       URL = {https://doi.org/10.1090/jams/1027},
}

@article {Talagrand1994randomgraph,
    AUTHOR = {Talagrand, Michel},
     TITLE = {On {R}usso's approximate zero-one law},
   JOURNAL = {Ann. Probab.},
  FJOURNAL = {The Annals of Probability},
    VOLUME = {22},
      YEAR = {1994},
    NUMBER = {3},
     PAGES = {1576--1587},
      ISSN = {0091-1798},
   MRCLASS = {28A35 (60K35)},
  MRNUMBER = {1303654},
       URL =
              {http://links.jstor.org/sici?sici=0091-1798(199407)22:3<1576:ORAZL>2.0.CO;2-S&origin=MSN},
}

@article {FKthreshold1996,
    AUTHOR = {Friedgut, Ehud and Kalai, Gil},
     TITLE = {Every monotone graph property has a sharp threshold},
   JOURNAL = {Proc. Amer. Math. Soc.},
  FJOURNAL = {Proceedings of the American Mathematical Society},
    VOLUME = {124},
      YEAR = {1996},
    NUMBER = {10},
     PAGES = {2993--3002},
      ISSN = {0002-9939},
   MRCLASS = {05C80},
  MRNUMBER = {1371123},
MRREVIEWER = {Andrzej Ruci\'{n}ski},
       DOI = {10.1090/S0002-9939-96-03732-X},
       URL = {https://doi.org/10.1090/S0002-9939-96-03732-X},
}

@article {CKLS2016,
    AUTHOR = {Chakraborty, Sourav and Kulkarni, Raghav and Lokam,
              Satyanarayana V. and Saurabh, Nitin},
     TITLE = {Upper bounds on {F}ourier entropy},
   JOURNAL = {Theoret. Comput. Sci.},
  FJOURNAL = {Theoretical Computer Science},
    VOLUME = {654},
      YEAR = {2016},
     PAGES = {92--112},
      ISSN = {0304-3975},
   MRCLASS = {94A17},
  MRNUMBER = {3574487},
       DOI = {10.1016/j.tcs.2016.05.006},
       URL = {https://doi.org/10.1016/j.tcs.2016.05.006},
}

@incollection {OWZ2011,
    AUTHOR = {O'Donnell, Ryan and Wright, John and Zhou, Yuan},
     TITLE = {The {F}ourier entropy-influence conjecture for certain classes
              of {B}oolean functions},
 BOOKTITLE = {Automata, languages and programming. {P}art {I}},
    SERIES = {Lecture Notes in Comput. Sci.},
    VOLUME = {6755},
     PAGES = {330--341},
 PUBLISHER = {Springer, Heidelberg},
      YEAR = {2011},
   MRCLASS = {94C10 (94A17)},
  MRNUMBER = {2874117},
MRREVIEWER = {C. Moraga},
       DOI = {10.1007/978-3-642-22006-7\_28},
       URL = {https://doi.org/10.1007/978-3-642-22006-7_28},
}

@inproceedings{KLW2010,
  title={Mansour's Conjecture is True for Random DNF Formulas.},
  author={Klivans, Adam R and Lee, Homin K and Wan, Andrew},
  booktitle={COLT},
  pages={368--380},
  year={2010},
  organization={Citeseer}
}

@incollection {OT2013,
    AUTHOR = {O'Donnell, Ryan and Tan, Li-Yang},
     TITLE = {A composition theorem for the {F}ourier entropy-influence
              conjecture},
 BOOKTITLE = {Automata, languages, and programming. {P}art {I}},
    SERIES = {Lecture Notes in Comput. Sci.},
    VOLUME = {7965},
     PAGES = {780--791},
 PUBLISHER = {Springer, Heidelberg},
      YEAR = {2013},
   MRCLASS = {68Q99 (06E30 94A17)},
  MRNUMBER = {3109120},
       DOI = {10.1007/978-3-642-39206-1\_66},
       URL = {https://doi.org/10.1007/978-3-642-39206-1_66},
}

@inproceedings {WWW2014,
    AUTHOR = {Wan, Andrew and Wright, John and Wu, Chenggang},
     TITLE = {Decision trees, protocols, and the {F}ourier entropy-influence
              conjecture},
 BOOKTITLE = {I{TCS}'14---{P}roceedings of the 2014 {C}onference on
              {I}nnovations in {T}heoretical {C}omputer {S}cience},
     PAGES = {67--79},
 PUBLISHER = {ACM, New York},
      YEAR = {2014},
   MRCLASS = {94A17 (68Q99 94C15)},
  MRNUMBER = {3359465},
MRREVIEWER = {Ulrich Tamm},
}

@article{das2011entropy,
  title={The entropy influence conjecture revisited},
  author={Das, Bireswar and Pal, Manjish and Visavaliya, Vijay},
  journal={arXiv preprint arXiv:1110.4301},
  year={2011}
}

@article{shalev2018fourier,
  title={On the Fourier Entropy Influence conjecture for extremal classes},
  author={Shalev, Guy},
  journal={arXiv preprint arXiv:1806.03646},
  year={2018}
}

@article {ACKSW21fourierentropy,
    AUTHOR = {Arunachalam, Srinivasan and Chakraborty, Sourav and Kouck\'{y}, Michal and Saurabh, Nitin and de Wolf, Ronald},
     TITLE = {Improved bounds on {F}ourier entropy and min-entropy},
   JOURNAL = {ACM Trans. Comput. Theory},
  FJOURNAL = {ACM Transactions on Computation Theory},
    VOLUME = {13},
      YEAR = {2021},
    NUMBER = {4},
     PAGES = {Art. 22, 40},
      ISSN = {1942-3454},
   MRCLASS = {94D10 (42A61 68Q11)},
  MRNUMBER = {4313286},
       DOI = {10.1145/3470860},
       URL = {https://doi.org/10.1145/3470860},
}

@article {KKLMS2020,
    AUTHOR = {Kelman, Esty and Kindler, Guy and Lifshitz, Noam and Minzer,
              Dor and Safra, Muli},
     TITLE = {Towards a proof of the {F}ourier-entropy conjecture?},
   JOURNAL = {Geom. Funct. Anal.},
  FJOURNAL = {Geometric and Functional Analysis},
    VOLUME = {30},
      YEAR = {2020},
    NUMBER = {4},
     PAGES = {1097--1138},
      ISSN = {1016-443X},
   MRCLASS = {68Q87 (06E30 94A17)},
  MRNUMBER = {4153910},
MRREVIEWER = {Wolfgang Lusky},
       DOI = {10.1007/s00039-020-00544-2},
       URL = {https://doi.org/10.1007/s00039-020-00544-2},
}

@article{kesal2017generalized,
  title={{Generalized Bounds on the Capacity of the Binary-Input Channels}},
  author={Kesal, Mustafa},
  journal={arXiv preprint arXiv:1710.06908},
  year={2017}
}

@article{wyner2003common,
  title={The common information of two dependent random variables},
  author={Wyner, Aaron},
  journal={IEEE Transactions on Information Theory},
  volume={21},
  number={2},
  pages={163--179},
  year={2003},
  publisher={IEEE}
}
\end{document}